\documentclass[11pt,bezier]{article}
\usepackage{amsmath, amssymb, amsfonts, graphicx}

\newtheorem{prethm}{{\bf Theorem}}

\newenvironment{thm}{\begin{prethm}\sl{\hspace{-0.5
               em}{\bf.}}}{\end{prethm}}

\newtheorem{prepro}[prethm]{{\bf Proposition}}

\newtheorem{prelem}[prethm]{{\bf Lemma}}

\newenvironment{lem}{\begin{prelem}\sl{\hspace{-0.5
               em}{\bf.}}}{\end{prelem}}

\newtheorem{predeff}[prethm]{{\bf Definition}}

\newtheorem{precor}[prethm]{{\bf Corollary}}

\newtheorem{preconj}[prethm]{{\bf Conjecture}}

\newenvironment{conj}{\begin{preconj}\sl{\hspace{-0.5
               em}{\bf.}}}{\end{preconj}}

\newtheorem{preremark}[prethm]{{\bf Remark}}

\newtheorem{preexample}[prethm]{{\bf Example}}

\newtheorem{preproof}{{\bf\textsf{Proof.}}}

\newenvironment{proof}[1]{\begin{preproof}{\rm
               #1}\hfill{$\Box$}}{\end{preproof}}

\newcommand{\la}{\lambda}

\title{  Proof of a conjecture on `plateaux' phenomenon of graph Laplacian eigenvalues}

\author{{\sc Ebrahim Ghorbani} \\[.3cm]
{\sl Department of Mathematics, K.N. Toosi University of Technology,}\\
{\sl P. O. Box 16315-1618, Tehran, Iran}\\
{\sl School of Mathematics, Institute for Research in Fundamental
Sciences (IPM),}\\
{\sl P.O. Box
19395-5746, Tehran, Iran }
\\[.3cm]
$\mathsf{e\_ghorbani@ipm.ir}$ }

\begin{document}
\maketitle

\vspace{5mm}

\begin{abstract}
Let $G$ be a simple graph. A pendant path of $G$ is a path  such that one of
its end vertices has degree $1$, the other end has degree $\ge3$, and all the internal vertices have degree $2$.  Let
$p_k(G)$ be the number of pendant paths of length $k$ of $G$, and $q_k(G)$ be the number of vertices with degree $\ge3$ which are an end vertex of some pendant paths of length $k$.
 Motivated by the problem of characterizing dendritic trees, N. Saito and E. Woei conjectured that any graph $G$ has some Laplacian eigenvalue with multiplicity at least $p_k(G)-q_k(G)$. 
We prove a more general result for both Laplacian and signless Laplacian eigenvalues from which the conjecture follows.

\vspace{5mm}
\noindent {\bf AMS Mathematics Subject Classification\,(2010):}   05C50, 05C38
\end{abstract}

\vspace{5mm}

\section{Introduction}
Let $G$ be a simple graph. A {\em pendant path} of $G$ is a path such that one of
its end vertices has degree $1$, the other end has degree $\ge3$, and all the internal vertices have degree $2$.  Let
$p_k(G)$ denote the number of pendant paths of length $k$ of $G$, and $q_k(G)$ denote the number of vertices with degree $\ge3$ which are an end vertex of some pendant paths of length $k$.
Saito and  Woei \cite{sw} studied Laplacian eigenvalues of dendritic trees. They observed  eigenvalue(s) `plateaux' (i.e., set of eigenvalue(s) with multiplicity) in the Laplacian eigenvalues of such trees. More generally, they showed that $(3\pm\sqrt5)/2$ is a Laplacian eigenvalue of any graph $G$ with multiplicity at least
$p_2(G)-q_2(G)$. This motivated them to put forward the following conjecture.

\begin{conj}\label{conj} {\rm(\cite{sw})} For any positive integer $k$, any graph $G$ has some Laplacian eigenvalue with multiplicity at least $p_k(G)-q_k(G)$.
\end{conj}

 We remark that the special cases $k=1$ follows from a result of \cite{f} (see also \cite{m}) asserting that multiplicity of $1$ as a Laplacian eigenvalue of a graph $G$ is at least  $p_1(G)-q_1(G)$.
We prove a more general result (Theorem~\ref{main} below) for both Laplacian and signless Laplacian eigenvalues from which Conjecture~\ref{conj} follows.

\section{Preliminaries}

Let $G$ be a simple graph with vertex set $ \{v_1, \ldots , v_n\}$ and edge set $ \{e_1, \ldots , e_m\}$.
 The {\em adjacency matrix} of $G$ is an $n \times  n$
 matrix $A=A(G)$ whose $(i, j)$-entry is $1$ if $v_i$ is adjacent to $v_j$ and  $0$ otherwise.
The {\em incident matrix} of
$G$, $X=X(G)=(x_{ij})$, is an $n\times m$ matrix whose rows and columns  are indexed by
vertex set and edge set of $G$, respectively, where
$$x_{ij}=\left\{\begin{array}{ll}
1 &  \hbox{if  $e_j$  is   incident with $v_i$,}\\
0 & \hbox{otherwise}.\\
\end{array}\right.$$
If we orient the edges of $G$, we may define similarly the {\em directed incidence matrix}  $D=D(G)=(d_{ij})$, with respect to the particular orientation,
 as
$$d_{ij}=\left\{\begin{array}{ll}
+1 & \hbox{if $e_j$ is  an incoming  edge to $v_i$,}\\
-1 & \hbox{if $e_j$ is  an outgoing  edge from $v_i$,}\\
0 & \hbox{otherwise.}
\end{array}\right.$$
The matrices $L(G):=DD^\top$ and $Q(G):=XX^\top$ are called the {\em Laplacian matrix} and {\em signless Laplacian matrix} of $G$, respectively.
 It is easily  seen that $L(G)=\Delta-A$  and $Q(G)=\Delta+A$
where $\Delta$ is the diagonal matrix of vertex degrees of $G$.

For any $n\times m$ matrix $M$, and nonzero real $\la$, we have
$$\det\begin{pmatrix}\la I_n&M\\M^\top&\la I_m\end{pmatrix}=\la^{m-n}\det(\la^2 I_n-MM^\top).$$
So we come up with the following lemma.
\begin{lem}\label{subdivid} For any graph $G$, $\la_1,\ldots,\la_r$ are all the nonzero eigenvalues of $Q(G)$ (resp. $L(G)$) if and only if $\pm\sqrt{\la_1},\ldots,\pm\sqrt{\la_r}$ are all the nonzero eigenvalues of $\begin{pmatrix}O&D\\D^\top&O\end{pmatrix}$ (resp. $\begin{pmatrix}O&X\\X^\top&O\end{pmatrix}$).
\end{lem}

 It is known that, $G$ is bipartite if and only if  $L(G)$ and $Q(G)$ are similar (see \cite[Prop.~1.3.10]{bh}).
 We will need the following variation of this result.
\begin{lem}\label{ddt} Let $G$ be a graph. Then $G$ is bipartite if and only if with respect to any orientation, any principal submatrix of
\begin{equation}\label{d}
\begin{pmatrix}O&D\\D^\top&O\end{pmatrix}
\end{equation}
is similar to the corresponding submatrix of \begin{equation}\label{x}\begin{pmatrix}O&X\\X^\top&O\end{pmatrix}.\end{equation}
\end{lem}
\begin{proof}{If \eqref{d} and \eqref{x} are similar, 
then they share the same nonzero eigenvalues, and  squaring those eigenvalues, by Lemma~\ref{subdivid}, 
lead to the nonzero eigenvalues of $L(G)$ and $Q(G)$. Now \cite[Prop.~1.3.10]{bh} implies that $G$ is bipartite.

For the converse, suppose that $G$ is bipartite with  bipartition $(V_1,V_2)$. For two different orientations of $G$, the corresponding matrices \eqref{d} are similar.
Because switching the orientation of any edge $e$ is equivalent to multiplying the column and the row  of \eqref{d} corresponding to $e$ by $-1$.
So we may assume that all the edges are oriented from $V_1$ to $V_2$.  Now we get \eqref{x} by multiplying \eqref{d} from left and right by the $\pm1$-diagonal matrix whose $-1$ entries correspond to the vertices in $V_1$. Since \eqref{x} and \eqref{d} are similar by a diagonal matrix, the corresponding principle submatrices of them are also similar.
}\end{proof}

The {\em subdivision} of $G$, denoted by $\widetilde{G}$, is the graph obtained by introducing $m$ new vertices to $G$, and replacing each edge $e_i=uv$, $i=1,\ldots,m$, by two new edges $uw_i$ and $w_iv$. It is easily seen that \eqref{x} gives $A(\widetilde{G})$. Also any orientation of $G$ induces a natural orientation on $\widetilde{G}$, where each $e_i=(u,v)$ is replaced by $(u,w_i)$ and $(w_i,v)$.  We denote this oriented graph by $\widetilde{G}_o$. The matrix \eqref{d} can be viewed as the (signed) adjacency matrix of $\widetilde{G}_o$.

\section{Main Result}

In this section we present the main result of the paper, from which Conjecture~\ref{conj} follows. We denote the path graph\footnote{Notice that here the definition of $P_k$ is different from the conventional notation where normally $P_k$ represents the path graph consisting of $k$ vertices.}of length $k$ by $P_k$.

\begin{thm}\label{main} Let $G$ be a graph. Then  $4\cos^2(\pi i/(2k+1))$ for any  $k\ge1$ and $i=1,\ldots,k$,
is both a Laplacian and a signless Laplacian eigenvalue of $G$ with multiplicity at least $p_k(G)-q_k(G)$.
\end{thm}
\begin{proof}{For simplicity we fix $k$, and write  $p$ and $q$ for $p_k(G)$ and $q_k(G)$.
 If $p-q=0$, there is nothing to prove, so we assume that $p-q\ge1$.
 Hence there are $p$ pendant paths 
 of length $k$ which are connected to the rest of the graph via $q$ vertices $v_1,\ldots,v_q$.

We first consider signless Laplacian eigenvalues of $G$. In view of Lemma~\ref{subdivid}, it suffices to consider $A(\widetilde{G})$.
Let $\la_1\ge\cdots\ge\la_n$ be the eigenvalues of $A(\widetilde{G})$.
The graph $\widetilde{G}\setminus\{v_1,\ldots,v_q\}$ contains $p$ connected components isomorphic to $P_{2k-1}$. We know that the eigenvalues of $A(P_{2k-1})$ consists of $\theta_i:=2\cos(\pi i/(2k+1))$ for $i=1,\ldots,2k$ (see \cite[p.~9]{bh}). 
This means that $p$ consecutive eigenvalues of $A(\widetilde{G}\setminus\{v_1,\ldots,v_q\})$, say $\eta_m,\ldots,\eta_{m+p-1}$, are all equal to $\theta_i$.
  By interlacing (see \cite[Corollary~2.5.2]{bh}), we have
  $$\theta_i=\eta_j\ge\la_{j+q}\ge\eta_{j+q}=\theta_i~~~~\hbox{for}~j=m,\ldots,m+p-q-1.$$
  It follows that $\theta_i$ is an eigenvalue of $A(\widetilde{G})$ with multiplicity at least $p-q$. So we are done by Lemma~\ref{subdivid}. Note that $\theta_i=-\theta_{2k+1-i}$.

We now turn to Laplacian eigenvalues.  Let $G'$ be the induced subgraph of $G$ by the vertices of theses $p$ paths. Note that $v_1,\ldots,v_q$ belong to  $G'$.
Suppose that the edges of $G$ have been oriented with associate directed incidence matrix $D=D(G)$ and   directed subdivision graph $\widetilde{G}_o$.
Consider $\widetilde{G}_o\setminus\{v_1,\ldots,v_q\}$.
This graph has at least $p+1$ connected components, $p$ of which are paths $P_{2k-1}$ with some orientations. Let $H$ be the union of these $p$ directed paths.
 Let $D_1$ and $X_1$ be the submatrices of $D(G')$ and $X(G')$, respectively, obtained by removing the rows corresponding to $v_1,\ldots,v_q$.
 Then
  \begin{equation}\label{d1}\begin{pmatrix}O&D_1\\D_1^\top&O\end{pmatrix}\end{equation}
   is the
 (signed) adjacency matrix $H$. As $G'$ is bipartite, by Lemma~\ref{ddt}, \eqref{d1} is similar to
 $$\begin{pmatrix}O&X_1\\X_1^\top&O\end{pmatrix}=A(pP_{2k-1}).$$
So each of $\theta_1,\ldots,\theta_{2k}$ is an eigenvalue of \eqref{d1} and so an eigenvalue of the (signed) $A(H)$ with multiplicity at least $p$.
By interlacing, each of $\theta_1,\ldots,\theta_{2k}$ is an eigenvalue of
$$\begin{pmatrix}O&D\\D^\top&O\end{pmatrix}$$
 with multiplicity at least $p-q$.
 Therefore, by Lemma~\ref{subdivid}, each of $\theta_1^2,\ldots,\theta_k^2$  is an eigenvalue of $DD^\top=L(G)$ with multiplicity at least $p-q$.
 }\end{proof}

\section*{Acknowledgments}
I would like to thank the referee for her/his comments which improved the presentation of the paper.
 The research of the author was in part supported by a grant from IPM (No. 94050114).

{}

\end{document}